\documentclass[11pt]{article}
\usepackage[centertags]{amsmath}
\usepackage[normal]{subfigure}
\usepackage{wrapfig}
\usepackage{verbatim}
\usepackage{amsfonts}
\usepackage{amssymb}
\usepackage{amsthm}
\usepackage{amsmath}
\usepackage{newlfont}
\usepackage{graphicx}
\usepackage{epsfig}
\usepackage{amsmath}
\usepackage{color}
\usepackage{multicol}
\usepackage{amssymb,amsmath}
\usepackage{latexsym}
\usepackage[normalem]{ulem}
\usepackage{cancel}
\usepackage{enumerate}
\usepackage{amsmath, amsthm, amssymb}
\usepackage{MnSymbol,wasysym}

\DeclareTextFontCommand{\rm}{\rmfamily}

\newtheorem*{remark}{Remarks}
\newtheorem{theorem}{Theorem}[section]

\newtheorem*{theorem*}{Euler's Finite Difference Theorem}
\newtheorem{exmp}{Example}
\newtheorem{proposition}[theorem]{Proposition}
\newtheorem{lemma}[theorem]{Lemma}

\usepackage{authblk}
\author[a]{Shahar Nevo}
\author[b]{Irina Raichik}

\affil[a,b]{Department of Mathematics, Bar-Ilan University, Ramat-Gan, 52900, Israel}
{
    \makeatletter
    \renewcommand\AB@affilsepx{: \protect\Affilfont}
    \makeatother

    \affil[ ]{Email addresses}

    \makeatletter
    \renewcommand\AB@affilsepx{, \protect\Affilfont}
    \makeatother

    \affil[a]{nevosh@math.biu.ac.il}
    \affil[b]{irina.raichik@gmail.com}
}

\title{New elementary formulas for any derivative of any rational function}
\date{}

\begin{document}
\maketitle

\begin{abstract}
In this paper we introduce elementary and completely explicit formulas for the derivative of any order of any function of the type
$\dfrac{1}{p}$, where $p$ is a polynomial with known zeros. These formulas lead easily to corresponding explicit formulas of any derivative of any rational function. The formulas are canonical in the sense of differentiation: they depend on some dummy parameters but do not become more complicated when the order of the derivative increases. As applications we get other elementary and explicit results: a solution to the general interpolation problem of Hermite, finding the primitive function for any rational function and getting an explicit formula for the $n$-th element of certain recursive sequences. We point out the possibility of producing general identities and in particular combinatorial identities by the formulas. We also discuss the very practical possibility of writing computer programs that will carry out some of the above applications.

\end{abstract}


\section{Introduction}

\noindent Given a polynomial of degree $n$,
\begin{equation} \label{one}
p(z)=z^{n}+a_{n-1}z^{n-1}+...+a_{0}
\end{equation}
it is very simple to find an explicit formula for its derivative of any order $t\geq 0$. These derivatives are of course identically zero for $t\geq n+1$.
\noindent However, if we introduce $p(z)$ in a multiplicative form,
$$p(z)=(z-z_{1})^{m_{1}}(z-z_{2})^{m_{2}}\cdots(z-z_{L})^{m_{L}}\;,$$
where $z_{1},z_{2},...,z_{L} $ are the distinct roots of $p(z)$ and $m_{i}$ are natural numbers, then a general formula for the derivative of order $t$ is much more complicated an it is based on the general Leibniz rule for differentiating the product of a finite number of functions
\begin{equation} \label{two}
\left ( f_{1}\cdot f_{2}\cdots f_{p}  \right )^{(m)}=\displaystyle\sum_{\substack{k_1+k_2+...+k_p=m\\k\geq0, k\in\mathbb{Z}}}f_{1} ^{(k_1)}f_{2} ^{(k_2)}\cdots f_{p}^{(k_p)} 
\end{equation}
This formula is indeed explicit, but its application depends on the number of ways to write a natural number $m$ as an ordered sum of $p$ non-negative integers, and so it is complicated. This is the situation also in other known explicit formulas we shall introduce in this paper.
\noindent The essence of this paper is to introduce (many) explicit and \textit{elementary} (but complicated) formulas for each derivative of any function of the form $h(z)=\dfrac{1}{p(z)}$, where $z$ is a complex variable. In contrast to the derivatives of a certain order of a polynomial, $p(z)$, these formulas are based on the multiplicative representation of $h(z)$, that is
\begin{equation} \label{three}
h(z)=\dfrac{1}{(z-z_{1})^{m_{1}}(z-z_{2})^{m_{2}}\cdots(z-z_{L})^{m_{L}}}\;,
\end{equation}
rather than the form $h(z)=\dfrac{1}{z^{n}+a_{n-1}z^{n-1}+...+a_{0}}$.

\noindent As a result of these formulas we can deduce formulas for the derivative of any order of any rational function, $R(z)=\dfrac{p_1(z)}{p_2(z)}$, by the basic Leibniz rule for the differentiation of product:
\begin{equation} \label{four}
\left ( f\cdot g \right )^{(n)}=\sum_{k=0}^{n}\binom{n}{k}f^{(k)}g^{(n-k)}
\end{equation}
(see \cite[p. 118]{Landau},  \cite[p. 170]{Meizler}).

\noindent Here we introduce $f=p_1$ in the form (\ref{one}) and $g=\dfrac{1}{p_2}$ in the form (\ref{three}). 

\noindent This paper is arranged as follows:\\
In Section 2 we introduce some known results about higher order derivatives.\\
In Section 3 we introduce our results on the formulas for derivatives of any order of functions of the form \ref{three}, which is the content of Theorem 3.1. In Section 4 we bring a few numerical examples of using these formulas. In Section 5 we show how to deduce one of these formulas from another one. In Section 6 we introduce some applications of the results. In Section 6.1 we explain how to get an explicit and elementary formula for the solution of the general interpolation problem of Hermite. In Section 6.2 we explain in general how to derive some identities by comparing the formulas we got in Theorem 3.1 to other known results. In section 6.3 we explain how to get combinatorial identities, (see \cite{Gould}), by Theorem 3.1. In Section 6.4 we show how to get explicit formulas for the primitive function of any rational function by involving Theorem 3.1 and the representation of such a function as a sum of partial fractions. We also show there how to find explicit formulas for the derivative of any order of any rational function. In Section 6.5 we show how to get by Theorem 3.1 explicit formulas for certain recursive sequences. In Section 6.6 we discuss the possibility of writing efficient computer programs in order to apply Theorem 3.1 for differentiation or integration of general rational functions and for the solution of the general interpolation problem.

\section{Higher order derivative}

\noindent Higher order derivatives for general functions (of one or more variables) and for certain classes of functions is an old subject in function theory. In addition to Leibniz formulas (2),(4) the most famous one is probably Cauchy's integral formula for the derivative of any order of an analytic function
$$
f^{(n)}(z)=\dfrac{n!}{2\pi i}\int_{\gamma}\dfrac{f(\zeta)d\zeta}{\left ( \zeta-z \right )^{n+1}}\,\; \;, n=0,1,2,...,
$$
\noindent see \cite[p. 120]{Ahlfors}, where $f$ is an analytic function in some domain $D$ in the complex plane $\mathbb{C}$ and $\gamma$
is a simple smooth and closed curve in $D$, whose interior lies in $D$ and $z$ is also in the interior of
 $\gamma$.

\noindent In the 19'th century the formula for the derivative of order $n$ of the composition of two sufficiently differentiable functions was invented, named after the Italian mathematician Fa\`a di Bruno.
\noindent It says
$$
\dfrac{\mathrm{d} ^{n}\left ( f\left ( g(x) \right ) \right )}{\mathrm{d} x^{n}}=\sum \dfrac{n!}{m_{1}!1!^{m_{1}}\cdot m_{2}!2!^{m_{2}}\cdot \cdot \cdot m_{n}!n!^{m_{n}}}\cdot f^{\left ( m_{1}+...+m_{n} \right )} \left ( g(x) \right )\cdot \prod_{j=1}^{n}\left ( g^{(j)}(x)) \right )^{m_j},
$$
where the summation is over all n-tuples of non-negative integers $\left ( m_1,...,m_n \right )$ satisfying the constraint $1\cdot m_1+2\cdot m_2+3\cdot m_3+...+n\cdot m_n=n$. For the history of Fa\`a di Bruno's formula see \cite{Craik}.

\noindent From recent years we have results about certain classes of functions, or specific types of higher order derivatives for general functions. For example, in \cite{SlevSaf}, the Slevinsky--Safouhi formula is introduced about iterated applying of the differential operator $\dfrac{\mathrm{d} }{x^{\mu}\mathrm{d} x}$. In \cite{KimSug}, the invariant Schwarzian derivative of higher order is discussed. In \cite{Sabatini}, higher order derivatives of period functions are discussed. We shall not go into details on these matters. Our results are much more elementary. So, returning to elementary functions, then in addition to polynomials, it is possible (and easy) to find a derivative of any order of
$x^{a},e^{x},\ln x$ and their linear combinations ($\sin x$ for example is a linear combination of $e^{ix},e^{-ix}$).

\noindent Among the functions of the form $h(z)=\dfrac{1}{p(z)}$, $p$ a polynomial, there is a simple formula for derivative of any order for $h(z)=\dfrac{1}{(z-a_1)(z-a_2)\cdots(z-a_L)}$ .

\noindent Indeed, by separating it into partial fractions we get $h(z)=\dfrac{A_1}{z-a_1}+\dfrac{A_2}{z-a_2}+\cdots+\dfrac{A_L}{z-a_L} \, ,$ where  $A_i=\displaystyle\prod_{\substack{j=1\\j\neq i}} 
 \frac{1}{a_i-a_j}$ for $1\leq i\leq L$. Now, since the $A_i$'s are constants it is easy to derive $h^{(t)} (z)$ of any desired order $t$. The next section deals with getting formulas of a different type for any derivative of functions of the form $h(z)=\dfrac{1}{p(z)}$. 

\section{The formulas for the derivative of any order of the reciprocal of a polynomial}

\noindent The formulas we introduce here in Theorem 3.1 are the heart of this paper, the source for all other results. They are elementary, but long. Their main addends are composed of 2 or 3 summations of pretty long terms.\\
In order to simplify this complexity, and to ease the use of these formulas in the sequel, we also write them in steps, according to the order of summation of each of their addends. For that we use some intermediate functions $\Psi_{1}(j);\;\Phi_{1}(j,r),\Phi_{2}(r);\;\Theta_{1}(j,r,p), \Theta_{2}(r,p)$ that will be defined in the statement of our theorem. For these intermediate functions we denote only their dependence on the indices of the summations $j, r, p$, but of course they depend also on the variable $z$ and the parameters $t, N, n_{1}, n_{2},...,n_{L}, s_{1}, s_{2},...,s_{L},s$. 

\begin{theorem}
Let 
\begin{equation} \label{five}
h(z)=\dfrac{1}{(z-a_1)^{n_1+1}(z-a_2)^{n_2+1}\cdots (z-a_L)^{n_L+1} } 
\end{equation} be a rational function such that $a_1,a_2,...,a_L$ are different complex numbers, and for any $1\leq i\leq L$, $n_i\geq 0$ is an integer. Let $z \in \mathbb{C}$ be a point which is not a pole of $h$. Let $s_{1}, s_{2},...,s_{L},s=s(z)$ be numbers in $\mathbb{C}$, such that $s\neq 0$ and $s_i=0$ if and only if $n_i=0$. \\
\noindent Then the following two collections of formulas hold:\\
%%%%%%%%%%%%%%%%%%%%%%%%%%%%%%%%%  FORMULA 2  %%%%%%%%%%%%%%%%%%%%%%%%%%%%%%%%

\noindent For any $t\geq 0$ an integer, and for every integer $N$, $N\geq t+1$:\\
\newline
\hspace*{-2.8cm}{\scriptsize ${h^{(t)}(z)=\dfrac{t!}{s^{t}}\displaystyle\sum_{j=1}^{N}  \left [  \frac{j^{N-t}(-1)^{N-j+1}}{j!(N-j)!}\displaystyle\sum_{r=1}^{L} \left \{  \dfrac{(-1)^{n_r}}{s_{r}^{n_r}(a_r-z)^{1+n_r+t}} \displaystyle\sum_{p=0}^{n_r}\Bigg \langle \dfrac{(js-ps_r)^{(\displaystyle\sum_{m=1}^{L}n_m)+L-1+t}(-1)^{p}}{p!(n_r-p)!)}\left (\displaystyle\prod_{\substack{m=1\\m\neq r}}^{L}\left ( \displaystyle\prod_{q=0}^{n_m} \dfrac{1}{(js-qs_m)(a_r-z)-(js-ps_r)(a_m-z)} \right )  \right )
\Bigg \rangle\right \} \right ]}$ \\
}
%%%%%%%%%%%%%%%%%%%%%%%%%%%%%%%%%%%%%%%%%%%%%%%%%%%%%%%%%%%%%%%%%%%%%%%%%%%%%%%%%
\newline
For every $t\geq 0$, if $\displaystyle\left (\sum_{m=1}^{L}n_m  \right )+t\geq 1$:\\
\newline
%%%%%%%%%%     F O R M U L A 1 ---------O P T I O N 4 %%%%%%%%%%%%%%%%%%%%%%%
\hspace*{-2.45cm}{\scriptsize $h^{(t)}(z)=\dfrac{1}{s^{t}}\left (  \displaystyle\prod_{r=1}^{L}\dfrac{1}{z-a_r} \right )\cdot(-1)^{(\displaystyle\sum_{m=1}^{L}n_m)+t}\displaystyle\sum_{\substack{r=1,\\n_r \geq 1}}^{L}\left (s_{r}^{\left (\displaystyle\sum_{\substack{m=1,\\m\neq r}}^{L} n_m \right )+t}\cdot\dfrac{1}{(z-a_r)^{n_r+t}} \displaystyle\sum_{p=1}^{n_r}\dfrac{p^{(\displaystyle\sum_{m=1}^{L}n_m)+t}}{p!(n_r-p)!} (-1)^{p}\left (\displaystyle\prod_{\substack{m=1,\\m\neq r}}^{L}\left ( \displaystyle\prod_{q=1}^{n_m}\dfrac{1}{qs_m(z-a_r)-ps_r(z-a_m)} \right ) \right ) \right )+$}\\
\hspace*{-2.cm}{\scriptsize $\dfrac{t!}{s^{t}}
\displaystyle\sum_{j=1}^{t}\left [ \left (\dfrac{(-1)^{t-j+1}}{j!(t-j)!}  \right )\sum_{r=1}^{L}\left \{ \frac{(-1)^{n_r}}{s_{r}^{n_r}(a_r-z)^{1+n_r+t}}\sum_{p=0}^{n_r}\left ( \dfrac{(js-ps_r)^{(\displaystyle\sum_{m=1}^{L}n_m)+L-1+t}}{p!(n_r-p)!} (-1)^{p}\left (\displaystyle\prod_{\substack{m=1,\\m\neq r}}^{L}\left ( \displaystyle\prod_{q=0}^{n_m}\dfrac{1}{(js-qs_m)(a_r-z)-(js-ps_r)(a_m-z)}\right ) \right ) \right ) \right \} \right ]$}\\
\newline
\noindent In a more compact and concise form, using auxiliary  functions  $\Psi_{1},\;\Phi_{1},\Phi_{2},\;\Theta_{1},\;\Theta_{2}$ we have respectively
to the above formulas

\begin{enumerate}[(I)]
  \item For any $t\geq 0$ an integer, and for every integer $N$, $N\geq t+1$\\
 $h^{(t)}(z)=\dfrac{t!}{s^{t}}(-1)^{N+t}\displaystyle\sum_{j=1}^{N}\Psi_{1}(j)$. Here for every $1\leq j\leq N$\\ $\Psi_{1}(j):=\dfrac{j^{N-t}}{j!(N-j)!}(-1)^{j} \displaystyle\sum_{r=1}^{L}\Phi_{1}(j,r)$ where for every $1\leq r\leq L$\\
$\Phi_{1}(j,r):=\dfrac{1}{s_{r}^{n_r}(z-a_r)^{n_r+1+t}}\displaystyle\sum_{p=0}^{n_r}\Theta_{1}(j,r,p)$
where for every $0\leq p\leq n_r$ \\
$\Theta_{1}(j,r,p):=\dfrac{(js-ps_r)^{(\displaystyle\sum_{m=1}^{L}n_m)+L-1+t}}{p!(n_r-p)!} (-1)^{p}\displaystyle\prod_{\substack{m=1,\\m\neq r}}^{L}\left ( \displaystyle\prod_{q=0}^{n_m}\dfrac{1}{(z-a_m)(js-ps_r)-(z-a_r)(js-qs_m)} \right )$\\
Thus, considering all the relations above we have 'in one piece'
%%%%%%%%%%%%%%%%%%%%%%%%%%%
%\begin{equation} \label{six}
%h^{(t)}(z)=\dfrac{t!}{s^{t}}(-1)^{N+t}\sum_{j=1}^{N} \left \langle %\dfrac{j^{N-t}}{j!(N-j)!}(-1)^{j}\sum_{r=1}^{L}\left %(\frac{1}{s_{r}^{n_r}(z-a_r)^{n_r+1+t}} \sum_{p=0}^{n_r}\Theta_{1}(j,r,p) \right ) %\right \rangle
%\end{equation}
%%%%%%%%%%%%%%%%%%%%%%%%%%
%%%%%%%%%%%%%%%%%%%%%%%%%%%
\begin{equation} \label{six}
h^{(t)}(z)=\dfrac{t!}{s^{t}}(-1)^{N+t}\sum_{j=1}^{N}\left (   \dfrac{j^{N-t}}{j!(N-j)!}(-1)^{j}\sum_{r=1}^{L}\left (\frac{1}{s_{r}^{n_r}(z-a_r)^{n_r+1+t}} \sum_{p=0}^{n_r}\Theta_{1}(j,r,p) \right )  \right )
\end{equation}
%%%%%%%%%%%%%%%%%%%%%%%%%%

with $\Theta_{1}(j,r,p)$ as above. 
  \item For every $t\geq 0$, if $\displaystyle\left (\sum_{m=1}^{L}n_m  \right )+t\geq 1$ then\\
  $h^{(t)}(z)=\dfrac{1}{s^{t}}\left ( \displaystyle\prod_{m=1}^{L}\dfrac{1}{z-a_m} \right )(-1)^{(\displaystyle\sum_{m=1}^{L}n_m )+t}\displaystyle\sum_{\substack{r=1,\\n_r \geq 1}}^{L}\Phi_2(r)+\dfrac{t!}{s^{t}}\displaystyle\sum_{j=1}^{t}\Psi_1(j)$\\
where for every $1\leq r\leq L,\; n_{r}\geq 1$,  $\Phi_2(r):=s_{r}^{\left (\displaystyle\sum_{\substack{m=1,\\m\neq r}}^{n_r} n_m \right )+t}\cdot\dfrac{1}{(z-a_r)^{n_r+t}}\displaystyle\sum_{p=1}^{n_r}\Theta_2(r,p)$,\\
where for every $1\leq p\leq n_r$, $\Theta_{2}(r,p):=\dfrac{p^{(\displaystyle\sum_{m=1}^{L}n_m)+t}}{p!(n_r-p)!} (-1)^{p}\displaystyle\prod_{\substack{m=1,\\m\neq r}}^{L}\left ( \displaystyle\prod_{q=1}^{n_m}\dfrac{1}{qs_m(z-a_r)-ps_r(z-a_m)} \right )$ \\
and $\Psi_1(j)$ is defined like in (\ref{six}) with $t$ instead of $N$ in the two places where $N$ appears. Thus, considering the relations above we have 'in one piece'
%%%%%%%%%%%%%%%%%%%%%%%%%%%%%%%%%%%%%%%%%%%%%%%%%%%%%%%%%%%%%%%%%%%%%%%%%%%%%%%%%%
%\begin{equation} \label{seven}
%\begin{aligned}
%& h^{(t)}(z)=\dfrac{1}{s^{t}}\left ( \displaystyle\prod_{m=1}^{L} %\dfrac{1}{z-a_m}\right )(-1)^{\left ( \displaystyle\sum_{m=1}^{L}n_{m} \right %)+t}\displaystyle\sum_{\substack{r=1,\\n_{r}\geq 1}}\left \langle s_{r}^{\left %(\displaystyle\sum_{{\substack{m=1,\\m\neq r}}}^{L}n_m \right )+t}\cdot %\frac{1}{(z-a_{r})^{n_r+t}}\sum_{p=1}^{n_{r}}\Theta_2(r,p)\right \rangle \\
%& +\dfrac{t!}{s^{t}}
%\displaystyle\sum_{j=1}^{t}\left \langle \frac{(-1)^{j}}{j!(t-j)!}\sum_{r=1}^{L}\left %\{ \frac{1}{s_{r}^{n_r}(z-a_r)^{n_r+1+t}}\sum_{p=0}^{n_r}\Theta_1(j,r,p) \right \} %\right \rangle
%\end{aligned}
%\end{equation}
%%%%%%%%%%%%%%%%%%%%%%%%%%%%%%%%%%%%%%%%%%%%%%%%%%%%%%%%%%%%%%%%%%%%%%%%%%%%%%%%%%%

%%%%%%%%%%%%%%%%%%%%%%%%%%%%%%%%%%%%%%%%%%%%%%%%%%%%%%%%%%%%%%%%%%%%%%%%%%%%%%%%%%
\begin{equation} \label{seven}
\begin{aligned}
& h^{(t)}(z)=\dfrac{1}{s^{t}}\left ( \displaystyle\prod_{m=1}^{L} \dfrac{1}{z-a_m}\right )(-1)^{\left ( \displaystyle\sum_{m=1}^{L}n_{m} \right )+t}\displaystyle\sum_{\substack{r=1,\\n_{r}\geq 1}}\left ( s_{r}^{\left (\displaystyle\sum_{{\substack{m=1,\\m\neq r}}}^{L}n_m \right )+t}\cdot \frac{1}{(z-a_{r})^{n_r+t}}\sum_{p=1}^{n_{r}}\Theta_2(r,p) \right ) \\
& +\dfrac{t!}{s^{t}}
\displaystyle\sum_{j=1}^{t}\left ( \frac{(-1)^{j}}{j!(t-j)!}\sum_{r=1}^{L}\left ( \frac{1}{s_{r}^{n_r}(z-a_r)^{n_r+1+t}}\sum_{p=0}^{n_r}\Theta_1(j,r,p) \right ) \right )
\end{aligned}
\end{equation}
%%%%%%%%%%%%%%%%%%%%%%%%%%%%%%%%%%%%%%%%%%%%%%%%%%%%%%%%%%%%%%%%%%%%%%%%%%%%%%%%%%%

Where $\Theta_{1}(j,r,p),\Theta_2(r,p)$ are as above. We note, that the relations explained before (\ref{six}) and (\ref{seven}) are efficient for easy calculation by these formulas respectively, and it is recommended to work out these formulas by these relations.
 \end{enumerate}
\end{theorem}
\begin{remark}
 \noindent
\begin{enumerate}[3.1]
\centerline{\uline{\bfseries Notation}}
  \item $h^{(0)}(z)$ means (as common) $h(z)$.
  \item $0^{0}$ is defined to be 1, for example $s_r^{n_r}$ if $n_r=0$.
  \item A sum with empty range is defined to be zero. For example $\displaystyle\sum_{j=1}^{t}\cdots$ if $t=0$ or $\displaystyle\sum_{\substack{r=1\\n_r\geq 1}}^{L} 
 \cdots$ if $n_r=0$ for every  $1\leq r\leq L$.
  \item  A product with empty range is defined to be $1$. For example $\displaystyle\prod_{\substack{m=1\\m\neq r}}^{L} 
 \cdots$ if $L=1$ or $\displaystyle\prod_{\substack{q=1}}^{n_m}\cdots$ if  $n_m=0$.\\
 \centerline{\uline{\bfseries Uniformity}}
 \item If $s_1,...,s_L,s$ are valid  (that is for a certain $z$, the denominators in the formulas for $h^{(t)}(z)$ in Theorem 3.1 are non-zero) for one of the formulas in Theorem 3.1, then they are valid for all other admissible formulas there.\\
 \centerline{\uline{\bfseries {Canonical choice of the dummy parameters}}}
 \item Evidently, almost any choice of $s_1,...,s_L,s$ is valid in the sense that the probability that this choice from the space $\displaystyle\mathbb{C}^{L+1}$ (with uniform distributed probability measure) is valid, is $1$. However, there are certain choices for which one or more of the denominators vanishes. Suppose we want to write a computer program (see Section 6.6) that uses these formulas. Then we must ensure that all the denominators are nonzero. This happens if and only if $\dfrac{a_r-z}{a_m-z}\neq \dfrac{js-ps_r}{js-qs_m} $, for all admissible values of $j,r,m,p,q$ in the formulas (\ref{six}) and (\ref{seven}).\\
 \noindent By elementary considerations it follows that these requirements can be simultaneously fulfilled in the following way.\\ 
 \noindent Set $M=M_z:=\displaystyle\prod_{\substack{1\leq m,r\leq L\\m\neq r}}\left ( 1+\left | \frac{a_r-z}{a_m-z} \right | \right )$ (a product with $L(L-1)$ factors). Now, without loss of generality we can assume that $n_i\geq 1$ for $1\leq i\leq L'$, and that $n_i=0$ for $L'+1\leq i\leq L$.\\
 \noindent Then define $s_{1}=1,\; s_{2}=\left ( 3M \right )n_{1},\; s_{3}=\left ( 3M \right )^{2}n_{2}\cdot n_{1},...,s_{L'}=\left ( 3M \right )^{L'-1}n_{L'-1}\cdot n_{L'-2}\cdots n_{1},\;\\ s=s(z):= \left ( 3M \right )^{L'}n_{L'}\cdot n_{L'-1}\cdots n_{1}$  and of course $s_{L'+1}=s_{L'+2}=\cdots =s_{L}=0$.\\
 \centerline{\uline{\bfseries {Connection to Euler's Finite Difference Theorem}}}
 \item Some terms of the formulas in Theorem 3.1 have some similarity to the terms in Euler's Finite Differences Theorem (see Sections 5 and 6.3).\\
 \noindent Indeed, in the course of the proof of Theorem 3.1 a proof of this theorem of Euler is deduced.\\
\centerline{\uline{\bfseries {Canonical meaning of the formulas in the sense of differentiation}}} 
\item The formulas for $h^{(t)}(z)$ are long (and similar to each other) and also depend on the dummy parameters $s_1,...,s_L,s$ and ion $N$. Nevertheless, they remain in the same form for any valid value of $t$, except for a change in $t$ itself. In this sense, these formulas can be considered as canonical for $h(z)$ in the sense of differentiation.
\end{enumerate}
\end{remark}

\section{Numerical examples of using the formulas of Theorem 3.1}

\noindent 
%%%%%%%%%%%%%%%%%%%%%%%%%%%%% E X A M P L E S %%%%%%%%%%%%%%%%%%%%%
\noindent We give here a few specific examples, in order to verify the validity of our formulas in different cases. The calculations are detailed in order to make the reading easier. The reader is invited to check the formulas in another cases.
\begin{exmp}
$h(z)=\dfrac{1}{z},\; t=0.$ Here $N=2$ is admissible for the formulas in (\ref{six}).We have $L=1,\; a_{1}=1,\; n_{1}=0$ and we get $h(z)=\dfrac{0!(-1)^{2}}{1}\cdot (-1)^{2}\left ( \Psi_1(j=1)+\Psi_1(j=2) \right )$.\\
$\Psi_1(j=1)=1\cdot (-1)^{1}\Phi_1(j=1,k=1)=-\dfrac{1}{0^{0}(z-0)^{1}}\cdot \Theta_1(j=1,r=1,p=0)=-\dfrac{1}{z}\dfrac{(s-0)^{0}}{0!0!}(-1)^{0}=-\dfrac{1}{z}$.\\
$\Psi_1(j=2)=\frac{2^{2}}{2!0!}(-1)^{2}\Phi_1(j=2,k=1)=2\dfrac{1}{0^{0}(z-0)^{1}}\cdot \Theta_1(j=2,r=1,p=0)=\dfrac{2}{z}\dfrac{(2s-0)^{0}}{0!0!}(-1)^{0}=\dfrac{2}{z}$.\\
Then we get: $\;h(z)=-\dfrac{1}{z}+\dfrac{2}{z}=\dfrac{1}{z}$.\\
Observe that by formula (\ref{seven}) we get $0$, but this formula is not admissible in this case, since  
$\displaystyle\sum_{m=1}^{1}n_m+t=0+0=0$.
\end{exmp}

\begin{exmp}
$h(z)=\dfrac{1}{z^{2}},\; t=0.$ Here $N=1$ is admissible. We have $L=1,\; n_{1}=1$ and we get by (\ref{six}) with some valid $s_1\neq 0$:\\
$h(z)=\dfrac{0!}{s^{0}}(-1)^{1+0}\Psi_1(j=1)=-\frac{1^{1}}{1!0!}(-1)^{1}\Phi_1(j=1,r=1)=\\
\dfrac{1}{s_{1}^{1}(z-0)^{1+1+0}}\left ( \Theta_1(j=1,r=1,p=0)
+\Theta_1(j=1,r=1,p=1) \right )$.\\
$\Theta_1(j=1,r=1,p=0)=\dfrac{(s-0)^{1}}{0!1!}(-1)^{0}\cdot 1=s,\; \Theta_1(j=1,r=1,p=1)=\dfrac{(s-s_1)^{1}}{1!0!}(-1)^{1}=s_{1}-s$ and we get\\
$h(z)=\dfrac{1}{s_{1}\cdot z^{2}}\left ( s+s_{1}-s \right )=\dfrac{1}{z^{2}}$, as required.\\
We can also use for this case formula (\ref{seven}). Here the second addend is $0$ and we get\\
$h(z)=h^{0}(z)=\dfrac{1}{s^{0}}\cdot\dfrac{1}{z}\cdot (-1)^{1+0}\cdot \Phi_2(1)=-\dfrac{1}{z}\cdot s_{1}^{0}\dfrac{1}{(z-0)^{1}}\cdot \Theta_2(r=1,p=1)=-\dfrac{1}{z^{2}}\cdot \dfrac{1^{1}}{1!0!}(-1)^{1}\cdot 1=\dfrac{1}{z^{2}}$\\
Here, in accordance with Remark (3.4) the double product in $\Theta_2(r=1,p=1)$ is taken to be $1$.
\end{exmp}
\noindent The rest of the examples are for specific values of $z$ and the parameters. This in order to simplify calculations.
\begin{exmp}
$h(z)=\dfrac{1}{z(z-2)^{2}}$ with $t=1,\; z=1$. Here $L=2,\; a_{1}=0,\; n_{1}=0$ (and thus $s_{1}=0$),\;$a_{2}=2, n_{1}=1$ and we take $s_{2}=-1$. We use formula from (\ref{six}) with $N=2$ to calculate $h'(1)$. We also take $s=s(z)=s(1)=3$.
As we shall see, the choice of $s,\;s_{1}$ is valid. We have\\
$h'(1)=\dfrac{1!}{3^{1}}(-1)^{2+1}\left ( \Psi_1(j=1)+\Psi_1(j=2) \right )$.\\
$ \Psi_1(j=1)=1\cdot (-1)\left ( \Phi_1(j=1,r=1) +\Phi_1(j=1,r=2)\right )$.\\
$\Phi_1(j=1,r=1)=\dfrac{1}{0^{0}(1-0)^{2}} \cdot \Theta_1(j=1,r=1,p=0)=\dfrac{1}{1^{2}}\dfrac{(3-0)^{3}}{0!0!}(-1)^{0}\cdot \dfrac{1}{(1-2)(3-0)-(1-0)(3-0)}\cdot \dfrac{1}{(1-2)(3-0)-(1-0)(3+1)}=27\cdot \left (-\dfrac{1}{6}  \right )\left (-\dfrac{1}{7}  \right )=\dfrac{9}{14}$.\\
$\Phi_1(j=1,r=2)=\dfrac{1}{(-1)^{1}(1-2)^{3}}\left ( \Theta_1(j=1,r=2,p=0)+\Theta_1(j=1,r=2,p=1) \right )$.\\
$\Theta_1(j=1,r=2,p=0)=\dfrac{(3-0)^{3}}{0!1!}(-1)^{0}\cdot\dfrac{1}{(1-0)(3-0)-(1-2)(3-0)}=27\cdot \dfrac{1}{6}=\dfrac{9}{2}$.\\
$\Theta_1(j=1,r=2,p=1)=\dfrac{(3+1)^{3}}{1!0!}\cdot (-1)^{1}\cdot \dfrac{1}{(1-0)(3+1)-(1-2)(3-0)}=64\cdot(-1)\cdot\dfrac{1}{7}=-\dfrac{64}{7}$. Then we get $\Phi_1(j=1,r=2)=1\cdot \left ( \dfrac{9}{2}-\dfrac{64}{7} \right )=-\dfrac{65}{14}$ and $\Psi_1(j=1)=-\left ( \dfrac{9}{14}-\dfrac{65}{14} \right )=4$.\\ Now to $\Psi_1(j=2)$:\\
$\Psi_1(j=2)=\dfrac{2^{1}}{2!0!}(-1)^{2}\left ( \Phi_1(j=2,r=1)+\Phi_1(j=2,r=2) \right )$.\\
$\Phi_1(j=2,r=1)=\dfrac{1}{0^{0}(1-0)^{2}}\cdot\Theta_1(j=2,r=1,p=0)=\dfrac{(6-0)^{3}}{0!0!}(-1)^{0}\cdot \dfrac{1}{(1-2)(6-0)-(1-0)(6-0)}\cdot \dfrac{1}{(1-2)(6-0)-(1-0)(6+1)}=216\cdot \dfrac{1}{(-12)}\cdot\dfrac{1}{(-13)}=\dfrac{18}{13} 
 $.\\
$\Phi_1(j=2,r=2)=\dfrac{1}{(-1)^{1}(1-2)^{3}}\left (\Theta_1(j=2,r=2,p=0)+\Theta_1(j=2,r=2,p=1)  \right )
$.\\
$\Theta_1(j=2,r=2,p=0)=\dfrac{(6-0)^{3}}{0!1!}(-1)^{0}\cdot \dfrac{1}{(1-0)(6-0)-(1-2)(6-0)}=216\cdot \dfrac{1}{12}=18$.\\
$\Theta_1(j=2,r=2,p=1)=\dfrac{(6+1)^{3}}{0!1!}(-1)^{1}\cdot \dfrac{1}{(1-0)(6+1)-(1-2)(6-0)}= -\dfrac{343}{13}$.\\
Then $\Phi_1(j=2,r=2)=18-\dfrac{343}{13}$ and $\Psi_1(j=2)=\left ( \dfrac{18}{13}+18-\dfrac{343}{13} \right )=-7$ and $h'(1)=-\frac{1}{3}(4-7)=1$. (This is true!)\\
The reader is invited to make the calculation with different valid values of the dummy parameters $s_2, s$.  Also, further work of calculations will bring us to the general formula for $h'(z)$.\\
Let us calculate again $h'(1)$. This time by (\ref{seven}), with $s_{2}=-6,\;s=s(1)=1$. We then get:\\
$h'(1)=\dfrac{1}{1^{1}}\cdot \dfrac{1}{1-0}\cdot \dfrac{1}{1-2}\cdot(-1)^{2}\cdot\Phi_2(r=2)+\dfrac{1!}{1^{1}}\Psi_1(j=1)$.\\
$\Phi_2(r=2)=(-6)^{1}\cdot \dfrac{1}{(1-2)^{2}}\Theta_2(r=2,p=1)=-6\cdot \dfrac{1^{1+1}}{1!0!}(-1)^{1}\cdot 1=6$.\\
$\Psi_1(j=1)=\dfrac{1^{0}}{1!0!}(-1)^{1}\left ( \Phi_1(j=1,r=1)+\Phi_1(j=1,r=2) \right )$.\\
$ \Phi_1(j=1,r=1)=\dfrac{1}{0^{0}(1-0)^{2}}\cdot \Theta_1(j=1,r=1,p=0)=1\cdot \dfrac{(1-0)^{3}}{0!0!}(-1)^{0}\cdot \dfrac{1}{(1-2)(1-0)-(1-0)(1-0)}\cdot\dfrac{1}{(1-2)(1-0)-(1-0)(1+6)}=1\cdot \dfrac{1}{(-2)}\cdot \dfrac{1}{(-8)}=\dfrac{1}{16}$.\\
$ \Phi_1(j=1,r=2)=\dfrac{1}{(-6)^{1}(1-2)^{3}}\cdot \left (\Theta_1(j=1,r=2,p=0)+ (\Theta_1(j=1,r=2,p=1)  \right )$.\\
$\Theta_1(j=1,r=2,p=0)=\dfrac{(1-0)^{3}}{0!1!}\cdot (-1)^{0}\cdot \dfrac{1}{(1-0)(1-0)-(1-2)(1-0)}=\dfrac{1}{2}$.\\
$\Theta_1(j=1,r=2,p=1)=\dfrac{(1+6)^{3}}{1!0!}\cdot (-1)^{1}\cdot \dfrac{1}{(1-0)(1+6)-(1-2)(1-0)}=\dfrac{343\cdot (-1)}{8}=-\dfrac{343}{8}$.\\
Then $\Phi_1(j=1,r=2)=\dfrac{1}{6}\left ( \dfrac{1}{2}-\dfrac{343}{8} \right )=-\dfrac{113}{6}$ and then 
$\Psi_1(j=1)=-\left ( \dfrac{1}{16}-\dfrac{113}{16} \right )=7$ and $h'(1)=-6+7=1$ as expected. Observe, that we cannot use, for example, the value we got for $\Phi_1(j=1,r=1)$ from the calculation by (\ref{six}), since the values of $s_{2}, s$ changed.\\
In Example 1 we showed that formula (\ref{seven}) does not work if $\displaystyle\left (\sum_{m=1}^{L}n_m  \right ) +t=0$. Let us show now a case in which the formulas of (\ref{six}) do not work when $N$ is out of the valid range. We demonstrate it with $h(z)=\dfrac{1}{z^{2}},\; t=1$ and take $N=t=1$. Here $L=1,\; a_{1}=0,\; n_{1}=1$ and we take $s_{1}=7$ and calculate for $z=2$ with $s=s(2)=2$.We then get by (\ref{six})\\
$\dfrac{1!}{2^{1}}(-1)^{2}\cdot \dfrac{1^{0}}{1!0!}(-1)^{1}\cdot \dfrac{1}{7^{1}(2-0)^{3}}\cdot \left ( \Theta_1(j=1,r=1,p=0)+\Theta_1(j=1,r=1,p=1) \right )$.\\
$\Theta_1(j=1,r=1,p=0)=\dfrac{(2-0)^{2}}{0!1!}(-1)^{0}=4$.\\
$\Theta_1(j=1,r=1,p=1)=\dfrac{(2-7)^{2}}{1!0!}(-1)^{1}=-25$ and we get in total $-\dfrac{1}{7\cdot 16}\left ( 4-25 \right )=\dfrac{3}{16}\neq -\dfrac{1}{4}=h'(2)$.

\end{exmp}

\section{Connections between the formulas in Theorem 3.1}

\noindent There is a strong connection between formulas (\ref{six}) and (\ref{seven}). We choosed to demonstrate it by showing how formula (\ref{six}) for the case $t=0$ implies formula (\ref{seven}) for the case $t=0$. For this we also need the following known theorem of Euler.
%%%%%%%%%%%%%%%%%%%%%%%%%%%%   Euler’s Finite Difference Theorem     %%%%%%%%%%%%%%%%%%%%%%
\begin{theorem*} 
\noindent Let $n,p$ be integers $0\leq p\leq n$. Then\\
$$\displaystyle\sum_{j=0}^{n}(-1)^{j}\binom{n}{j}j^{p}=\left\{\begin{matrix}
0 & \; \; \; \: \: \: \: 0\leq p\leq n\\ 
(-1)^{n}\cdot n! &p=n 
\end{matrix}\right.$$
\end{theorem*}
\noindent Observe that when $p>0$ then $j=0$ can be omitted. When $p=0$ we get the known fact that the alternate sum of the binomial coefficients is zero (here $0^{0}=1$).
\begin{proposition}
Formula (\ref{six}) for the case $t=0$ with every admissible $N$ implies formula (\ref{seven}) for the case $t=0$.
\end{proposition}
 
\begin{proof}
We apply (\ref{six}) with some $N,\;N \geq 0+1$. We separate $\displaystyle\sum_{p=0}^{n_{r}}\Theta_1(\cdot )$ to the addend of $p=0$ and to $\displaystyle\sum_{p=1}^{n_{r}}\Theta_1(\cdot )$. We then have\\ 
%%%%%%%%%%%%%%%%%%%%%%%%%%%%%%%%%%%%%%%%%%%%%%%%%%%%%%%%%%%%%%%%%%%%%%%%%%%%%%%%
%$h(z)=h^{0}(z)=(-1)^{N}\displaystyle\sum_{j=1}^{N}  \left \langle %\dfrac{j^{N}}{j!(N-j)!}(-1)^{j}\displaystyle\sum_{r=1}^{L}\left \{ %\frac{1}{s_{r}^{n_r}(z-a_r)^{n_r+1}}\left ( %\Theta_1(j,r,0)+\displaystyle\sum_{p=1}^{n_r}\Theta_1(j,r,p) \right ) \right \}  %\right \rangle$.
%%%%%%%%%%%%%%%%%%%%%%%%%%%%%%%%%%%%%%%%%%%%%%%%%%%%%%%%%%%%%%%%%%%%%%%%%%%%%%%
$h(z)=h^{0}(z)=(-1)^{N}\displaystyle\sum_{j=1}^{N}  \left ( \dfrac{j^{N}}{j!(N-j)!}(-1)^{j}\displaystyle\sum_{r=1}^{L}\left ( \frac{1}{s_{r}^{n_r}(z-a_r)^{n_r+1}}\left ( \Theta_1(j,r,0)+\displaystyle\sum_{p=1}^{n_r}\Theta_1(j,r,p) \right ) \right )  \right )$ 
\noindent Recall that the parameter $s$ is \textit{hidden} inside the $\Theta_1(\cdot )$'s. Since the denominators in (\ref{seven}) are nonzero by the choice of the parameters, we can let $s \to 0$. Since $\displaystyle\sum_{m=1}^{L}n_m \geq 1$ we get that $\Theta_1(j,r,0) \underset{s\to 0 }{\rightarrow 0}$ for every relevant $j,r$ (consider especially the case when $q=0$ in the double product in $\Theta_1(j,r,0)$). For $1\leq p\leq n_r$, we separate in the double product of $\Theta_1(j,r,p)$ the factors of $q=0$ from others. We get that the limit of that double product is \\
$\left (\displaystyle\prod_{\substack{m=1\\m\neq r}}^{L}\left ( \displaystyle\prod_{q=1}^{n_m} \dfrac{1}{(z-a_m)(-ps_r)-(z-a_r)(-qs_m)} \right )  \right )\displaystyle\prod_{\substack{m=1\\m\neq r}}^{L}\left ( \dfrac{1}{(z-a_m)(-ps_r)}\right )$. Note that we consider here only $r$ such that $n_r \geq 1$ (see Remark 3.3), so $s_r \neq 0$ in the last product. Also, by the choice of the parameters, the denominators here are nonzero and the limit  indeed exists.\\
Overall we get in the limit that:\\ 
$\Theta_1(j,r,p) \underset{s\to 0 }{\rightarrow }\dfrac{(-ps_r)^{\left ( \displaystyle\sum_{m=1}^{L}n_m \right )+L-1}}{p!(n_r-p)!}\cdot (-1)^{p}\cdot \dfrac{1}{(-ps_r)^{L-1}}\displaystyle\prod_{\substack{m=1\\m\neq r}}^{L}\left ( \frac{1}{z-a_m} \right )\left (\displaystyle\prod_{\substack{m=1\\m\neq r}}^{L}\left ( \displaystyle\prod_{q=1}^{n_m} \frac{1}{qs_m(z-a_r)-ps_r(z-a_m)} \right )  \right )$\\
$=(-1)^{\displaystyle\sum_{m=1}^{L}n_m}\cdot \dfrac{p^{\displaystyle\sum_{m=1}^{L}n_m}}{p!(n_r-p)!}\cdot (-1)^{p}s_r^{\displaystyle\sum_{m=1}^{L}n_m}\left (\displaystyle\prod_{\substack{m=1\\m\neq r}}^{L}\left ( \displaystyle\prod_{q=1}^{n_m} \dfrac{1}{qs_m(z-a_r)-ps_r(z-a_m)} \right )  \right )\displaystyle\prod_{\substack{m=1\\m\neq r}}^{L} \frac{1}{z-a_m}$\\
$=(-1)^{\displaystyle\sum_{m=1}^{L}n_m}\cdot s_r^{\displaystyle\sum_{m=1}^{L}n_m}\displaystyle\left (\prod_{\substack{m=1\\m\neq r}}^{L} \frac{1}{z-a_m}   \right )\Theta_2(r,p)$, and this is of course the exact value of $\Theta_1(j,r,p)$, since as long as the parameter $s$ is valid its value does not matter. We then get:\\
%%%%%%%%%%%%%%%%%%%%%%%%%%%%%%%%%%%%%%%%%%%%%%%%%%%%%%%%%%%%%%%%%%%%%%%%%%%%%%%%%%
%$h(z)=(-1)^{N}\displaystyle\sum_{j=1}^{N}  \left \langle \dfrac{j^{N}}{j!(N-j)!}(-1)^{j}\displaystyle\sum_{r=1}^{L}\left \{ \frac{1}{s_{r}^{n_r}(z-a_r)^{n_r+1}}\displaystyle\sum_{p=1}^{n_r}\left ((-1)^{\displaystyle\sum_{m=1}^{L}n_m}\cdot s_r^{\displaystyle\sum_{m=1}^{L}n_m}\displaystyle\left (\prod_{\substack{m=1\\m\neq r}}^{L} \dfrac{1}{z-a_m}   \right )\Theta_2(r,p) \right ) \right \}  \right \rangle $\\
%$=(-1)^{N}\displaystyle\sum_{j=1}^{N}  \left \langle %\dfrac{j^{N}}{j!(N-j)!}(-1)^{j}\displaystyle\sum_{r=1}^{L}\left \{ \frac{1}{(z-a_r)^{n_r}}\left (\displaystyle\prod_{m=1}^{L} \frac{1}{z-a_m} \right )(-1)^{\displaystyle\sum_{m=1}^{L}n_m}\cdot s_r^{\displaystyle\sum_{\substack{m=1\\m\neq r}}^{L}n_m}\displaystyle\sum_{p=1}^{n_r} \Theta_2(r,p)  \right \}  \right \rangle 
%$\\
%%%%%%%%%%%%%%%%%%%%%%%%%%%%%%%%%%%%%%%%%%%%%%%%%%%%%%%%%%%%%%%%%%%%%%%%%%%%%%%%%%%%
$h(z)=(-1)^{N}\displaystyle\sum_{j=1}^{N}  \left ( \dfrac{j^{N}}{j!(N-j)!}(-1)^{j}\displaystyle\sum_{r=1}^{L}\left ( \frac{1}{s_{r}^{n_r}(z-a_r)^{n_r+1}}\displaystyle\sum_{p=1}^{n_r}\left ((-1)^{\displaystyle\sum_{m=1}^{L}n_m}\cdot s_r^{\displaystyle\sum_{m=1}^{L}n_m}\displaystyle\left (\prod_{\substack{m=1\\m\neq r}}^{L} \dfrac{1}{z-a_m}   \right )\Theta_2(r,p) \right ) \right )  \right ) $\\
$=(-1)^{N}\displaystyle\sum_{j=1}^{N}  \left( \dfrac{j^{N}}{j!(N-j)!}(-1)^{j}\displaystyle\sum_{r=1}^{L}\left ( \frac{1}{(z-a_r)^{n_r}}\left (\displaystyle\prod_{m=1}^{L} \frac{1}{z-a_m} \right )(-1)^{\displaystyle\sum_{m=1}^{L}n_m}\cdot s_r^{\displaystyle\sum_{\substack{m=1\\m\neq r}}^{L}n_m}\displaystyle\sum_{p=1}^{n_r} \Theta_2(r,p)  \right ) \right )
$\\
%We got rid of the dependence in $s$. It caused also to get rid of the dependence in %$j$ inside $\displaystyle\sum_{r=1}^{L}\left \{ \cdots \right \}$, and so we can %change the order of the summations 
%$\displaystyle\sum_{j=1}^{N}  \left \langle \cdots \right \rangle$ and %$\displaystyle\sum_{r=1}^{L}\left \{ \cdots \right \}$, or more accurately to consider %$\displaystyle\sum_{r=1}^{L}\left \{ \cdots \right \}$ as a constant with respect to %$\displaystyle\sum_{j=1}^{N}  \left \langle \cdots \right \rangle$.\\
We got rid of the dependence in $s$. It caused also to get rid of the dependence in $j$ inside $\displaystyle\sum_{r=1}^{L}\left ( \cdots \right )$, and so we can change the order of the summations 
$\displaystyle\sum_{j=1}^{N}  \left ( \cdots \right )$ and $\displaystyle\sum_{r=1}^{L}\left ( \cdots \right )$, or more accurately to consider $\displaystyle\sum_{r=1}^{L}\left ( \cdots \right )$ as a constant with respect to $\displaystyle\sum_{j=1}^{N}  \left ( \cdots \right )$.\\
Now, by Euler's Finite Difference Theorem $(-1)^{N}\displaystyle\sum_{j=1}^{N}\dfrac{j^{N}}{j!(N-j)!}=1 $ and we get \\
$ h(z)=\left(\displaystyle\prod_{m=1}^{L} \dfrac{1}{z-a_m} \right )(-1)^{\displaystyle\sum_{m=1}^{L}n_m}\cdot\displaystyle\sum_{r=1}^{L} s_r^{\displaystyle\sum_{\substack{m=1\\m\neq r}}^{L}n_m}\cdot\dfrac{1}{(z-a_r)^{n_r}}\displaystyle\sum_{p=1}^{n_r} \Theta_2(r,p)$ \\

\noindent Of course we consider here only values of the index $r$ for which $n_r \geq 1$.
\noindent Now, since the second main addend in (\ref{seven}) is zero when $t=0$ (empty summation), we get exactly the formula (\ref{seven}) for $t=0$, as desired. 

\end{proof}

\section{Applications of Theorem 3.1}

\subsection{Interpolation}
\noindent An important application we have found for Theorem 3.1 is an explicit and $\textit{elementary}$ formula for the solution of the general interpolation problem of Hermite.
\noindent Let us first introduce the problem.\\

\noindent Let $a_{1},a_{2},...,a_{L}$ be different points in $\mathbb{C}$, and $n_{1},n_{2},...n_{L}$ non-negative integers. For each $1\leq i\leq L$ corresponds $n_{i}+1$ numbers $A_{0}^{(i)},A_{1}^{(i)},...,A_{n_{i}}^{(i)}$. Find a polynomial $p(z)$ of minimal degree such that  $p^{(l)}\left ( a_{i} \right )=A_{l}^{(i)}$, for every $1\leq i\leq L$, $0\leq l\leq n_{i}$. It is known by the theory of interpolation that there exist exactly one such polynomial of degree at most $n=\displaystyle\sum_{i=1}^{L}n_{i}+L-1$. The first case of this interpolation problem that was solved is the case where $n_{r}=0$ for $1\leq r\leq L$, and is due to Lagrange. Then the required polynomial is $p(z)=\displaystyle\sum_{i=1}^{L}l_{L,i}(z) \cdot A_{0}^{(i)}$, where 
\begin{equation} \label{eight}
l_{L,i}(z)=\displaystyle\prod_{\substack{i=1\\j\neq i}}^{L} \dfrac{z-a_{j}}{a_{i}-a_{j}},
\end{equation}
(see \cite[p. 7]{Kaniel},\cite{Spitzbart}).

\noindent Hermite (\cite[p. 11]{Kaniel},\cite{Spitzbart}) extended the solution to the case that $1=n_{1}=n_{2}=...=n_{L}$, and got the minimal polynomial of degree at most $2L-1$, $p(z)=\displaystyle\sum_{i=1}^{L}\varphi_{L,i}(z)A_{0}^{(i)}+\displaystyle\sum_{i=1}^{L}\psi_{L,i} (z)A_{1}^{(i)}$, where for each $1\leq i\leq L$, $\varphi_{L,i}$ and $\psi_{L,i}$ are defined as follows (see  \cite[p. 12]{Kaniel})
$$\varphi_{L,i}(z)=\left ( 1-2l'_{L,i}(a_i)(z-a_{i}) \right )l_{L,i}^{2}(z),$$
$$\psi_{L,i}(z)=\left ( z-a_{i} \right )^{2}l_{L,i}^{2}(z).$$

\noindent Spitzbart \cite{Spitzbart}, introduced formula for the general case as follows. Define for each $1\leq i\leq L$ a polynomial $$p_{i}(z):=(z-a_{1})^{n_{1}+1}\cdot(z-a_{2})^{n_{2}+1}\cdots(z-a_{i-1})^{n_{i-1}+1}\cdot(z-a_{i+1})^{n_{i+1}+1}\cdot(z-a_{i+2})^{n_{i+2}+1}\cdots(z-a_{L})^{n_{L}+1},$$ 
and set $g_{i}(z)=\dfrac{1}{p_{i}(z)}$. Then the solution to the interpolation problem is $p(z)=\displaystyle\sum_{i=1}^{L}\displaystyle\sum_{l=0}^{n_{i}}Q_{i,l}(z)A_{l}^{(i)}$, where $Q_{i,l}(z)=p_{i}(z)\cdot \dfrac{(z-a_{i})^{l}}{l!}\displaystyle\sum_{t=0}^{n_{i}-l}\dfrac{g_{i}^{(t)}(a_{i})}{t!}(z-a_{i})^{t}$.

\noindent This is an explicit formula, but not elementary, since the values $g_{i}^{(t)}$ are yet to be calculated. Theorem 3.1 gives an elementary formulas exactly for this need. It seems that the $\textit{'minimal'}$
case of the formulas (\ref{six}) i.e. with $N=t+1$ or formula (\ref{seven}) are most simple for this case
(and in general to calculate some high order derivative).

\subsection{General identities}
\noindent Over the years,there were made attempts to find easy formulas for solution of the general interpolation problem. In \cite[pp. 306-308]{LanTis} and in \cite{SakVer} the solution is expressed by recursive process. In \cite{HicYang} $p(z)$ is expressed by what known as \textit{cycle index polynomials} of certain symmetric groups. In \cite{KeDeTsPe} they got the formula $p(z)=\displaystyle\sum_{i=1}^{L}\displaystyle\sum_{l=0}^{n_i}Z_{i}\left ( I_{n_{i}+1}-\Lambda_{i}  \right )^{l} A_{i}$, where each $Z_{i}$ is a row vector depends on the variable $z$ of size $n_{i}+1$, and $\Lambda_{i}$ is a lower triangular matrix of constants of size $(n_{i}+1)\times (n_{i}+1)$. It is possible to relate our formulas (\ref{six}) and (\ref{seven}) to all these formulas mentioned above to get some identities. Of course we shall not detail here these matters. 

\subsection{Combinatorial identities}

\noindent By taking $h(z)=\dfrac{1}{z}$  in any of the many formulas in (\ref{six}) or in (\ref{seven}), the case $p=n$ in Euler's Finite Difference Theorem follows easily. In addition, fixing some function $h(z)$ of the form (\ref{five}) and some $z=z_{0}$, the right hands of formulas (\ref{six}) and (\ref{seven}) are constant functions of the $L'+1$ variables $s_1,...,s_{L'},s$ (see Remarks 3.1, 3.6). Hence any partial derivative of these expressions with respect to any of these variables is zero. This gives the ability to produce many combinatorial identities. Other combinatorial identities can be obtained by comparing any two of the formulas, for the same value of $t$, the same function $h$ and the same point $z$. The residue Theorem can also be used. For a certain function $h(z)$ of the type (\ref{five}), the only poles are $a_{1},a_{2},...,a_{L}$. But a typical pole of the right hand of the formulas in (\ref{six}) or (\ref{seven}) is a zero $z$ of some denominator, that is some $z$ for which
\begin{equation} \label{nine}
(js-qs_{m})(a_{r}-z)-(js-ps_{r})(a_{m}-z)=0 \;,
\end{equation}
or
\begin{equation} \label{ten}
 z=\frac{(js-qs_{m})a_{r}-(js-ps_{r})a_{m}}{ps_{r}-qs_{m}} \;,
\end{equation}
for the relevant values of $j,p,q,r,m$.

\noindent Thus, the sum of coefficients of any term as in (\ref{nine}) must be zero, otherwise $h$ would have an additional pole, except $a_{1},a_{2},...,a_{L}$. This is true even if the expression in (\ref{ten}) is equal to some $a_{i},\;1\leq i \leq L$, by the uniqueness theorem for meromorphic functions.  

\noindent We can also differentiate each of the formulas for $h^{(t)}(z)$ with respect to $z$ and get a formula for $h^{(t+1)}(z)$. Of course, checking all these possibilities, or better to say - selecting the
reasonable ones among them is an hard mission.

\subsection{Integration and derivatives of general rational function}
\noindent For a comprehensive text about the material of this section, the reader is referred to \cite[pp. 284-290]{Meizler} and \cite[pp. 243-269]{Landau}.
\noindent It is known that $h(z)$ as in (\ref{five}) can be represented as a finite sum of partial fractions as follows:
\begin{equation} \label{eleven}
h(z)=\displaystyle\sum_{i=1}^{L}\displaystyle\sum_{j=1}^{n_i+1}\dfrac{B_{j}^{(i)}}{(z-a_i)^{j}}\;,
\end{equation}
where the ${B_{j}^{(i)}}$'s are uniquely determined constants. Finding these constants is involved with reconstructing the common denominator of the partial fractions in (\ref{eleven}) (which is the denominator of $h$ from (\ref{five})). That leads to use the method of comparing of coefficients. It is known and easy to verify that each ${B_{j}^{(i)}}$ is given by
\begin{equation} \label{twelve}
{B_{j}^{(i)}}=\dfrac{\left ( (z-a_i)^{n_i+1}h(z) \right )^{(n_i+1-j)}}{(n_i+1-j)!}(a_i)
\end{equation}
Now, in this relation, the function that we have to differentiate, $(z-a_i)^{n_i+1}h(z)$, is exactly of the form (\ref{five}), and $a_i$ is not a pole of it. So we can express explicitly (and elementary) ${B_{j}^{(i)}}$ by the formulas (\ref{six}) and (\ref{seven}). After we have the constants ${B_{j}^{(i)}}$ , the representation (\ref{eleven}) of $h(z)$ enables us to calculate explicitly a primitive function of $h$. It is also very useful to calculate a complex integral $\dint_{\gamma}h(\zeta)d\zeta$ for some closed curve $\gamma$, that does not contain any point $a_i,\; 1\leq i\leq L$. Indeed, we know all the residues, 
${B_{1}^{(i)}},\;1\leq i\leq L$.
\noindent A similar representation to (\ref{eleven}) can be constructed for any rational function $H(z)$. Indeed, let $H(z)=\dfrac{p(z)}{q(z)}$, where $p,q$ are polynomials and $q(z)=(z-a_{1})^{n_{i}+1}\cdot(z-a_{2})^{n_{2}+1}\cdots(z-a_{L})^{n_{L}+1}$ where as usual each $n_i$ is a non-negative integer.
\noindent  By dividing in \textit{long division} $p$ by $q$, we can write $H$ as $H(z)=b(z)+\dfrac{\hat{p}(z)}{q(z)}$, where $b,\hat{p}$ are polynomials with $|\hat{p}|<|q|$. Then the rational function $\hat{h}(z)=\dfrac{\hat{p}(z)}{q(z)}$ can also be written as in (\ref{eleven}), say 
$\hat{h}(z)=\displaystyle\sum_{i=1}^{L}\displaystyle\sum_{j=1}^{n_i+1}\dfrac{\hat{B}_{j}^{(i)}}{(z-a_i)^{j}}$, where the $\hat{B}_{j}^{(i)}$'s are constants, and similarly to (\ref{twelve}) we have: 
\begin{equation} \label{thirteen}
{\hat{B}_{j}^{(i)}}=\dfrac{\left ( (z-a_i)^{n_i+1}\hat{h}(z) \right )^{(n_i+1-j)}}{(n_i+1-j)!}(a_i)\\
\end{equation}
for each $1\leq i\leq L,\;1\leq j\leq n_i+1$. So we have to find here the derivative of order $n_i+1-j$ of the function $\left (\displaystyle\prod_{\substack{k=1\\k\neq i}}^{L}\dfrac{1}{(z-a_k)^{n_k+1}}   \right )\hat p(z)$ at $a_i$, and we can find it explicitly by the formulas in (\ref{six}) and (\ref{seven}), and by Leibniz rule (\ref{four}).

\noindent In general, if $R(z)=\displaystyle\frac{P(z)}{Q(z)}$ is a rational function, where $\displaystyle\frac{1}{Q}$ is of the form (\ref{five}), then for $t\geq 0$ we have by (\ref{four}) $R^{(t)}(z)=\displaystyle\sum_{j=0}^{t}\binom{t}{j}P^{(j)}\left (1/Q  \right )^{(t-j)}(z)$. Now, if $|P|$ is the degree of $P$ and $j>|P|$, then $P^{(j)}(z)\equiv0$. Hence, if $t\geq |P|$, the above summation runs only for $j=0,1,2,...,P$ and then by (\ref{six}) and (\ref{seven}) we get explicit formulas for 
$R^{(t)}(z)$.

\subsection{Recursive sequences}

\noindent We want to construct certain infinite sequence $\left \{ b_n \right \}_{n=0}^{\infty }$ in a recursive way. First let us choose 4 complex numbers: $b_0, b_1,b_2, b_3$ (the first four elements of the sequence). In addition let $c_0, c_1, c_2, c_3$  be the other fixed 4 complex numbers, with $c_3\neq 0$, and let be $b_5, b_6, b_7,...$ be defined by the rule:
\begin{equation} \label{fourteen}
b_{n+4}=c_{0}b_{n+3}+c_{1}b_{n+2}+c_{2}b_{n+1}+c_{3}b_{n},\; n\geq 0
\end{equation}
We then get an infinite sequence (that is some kind of a generalization to the well known Fibonacci's sequence $b_{n+2}=b_{n+1}+b_{n}$). Let us show how to find an explicit formula to the elements $\left \{ b_n \right \}_{n=4}^{\infty }$. For this purpose consider the series:
\begin{equation} \label{fifteen}
\displaystyle\sum_{n=0}^{\infty}b_{n}z^{n}
\end{equation}
(see \cite[p. 184 Ex. 5]{Ahlfors} for comparison). We need 
\begin{lemma}
The series (\ref{fifteen}) has a positive radius of convergence at $z_0=0$.
\end{lemma}

\begin{proof}
Let $M=\max\left \{ |b_0|,|b_1|,|b_2|,|b_3| \right \},\; T=\max\left \{ |c_0|,|c_1|,|c_2|,|c_3| \right \}$ and define a sequence $\left \{ \hat{b}_n \right \}_{n=0}^{\infty }$ by setting  $\;\hat{b}_0=\hat{b}_1=\hat{b}_2=\hat{b}_3=M$ and for $n\geq 0$ define $\hat b_{n+4}:=\left ( T+1 \right )\left (\hat b_{n+3}+\hat b_{n+2}+\hat b_{n+1}+\hat b_{n}   \right )$. Evidently $|b_n|\leq \hat b_{n}$
for every $n \geq 0$. Also, since $T+1>1$, $\left \{ \hat{b}_n \right \}_{n=0}^{\infty }$ is a non-decreasing sequence and we have $\hat b_{n+1}\leq 4\left ( T+1 \right )\hat{b}_n$ for $n \geq 3$. Hence, for every $n \geq 4$ we have $\hat b_{n}\leq \left (4\left ( T+1 \right )  \right )^{n-3}\hat{b}_3=\left (4\left ( T+1 \right )  \right )^{n-3}M$.  
\noindent Thus $\overline{\lim}\sqrt[n]{\hat b_{n}}\leq 4(T+1)$ and so we also have that $\overline{\lim}\sqrt[n]{|b_n|}\leq 4(T+1)$. Thus we get that the radius of convergence $R$ of the series (\ref{fifteen}) is greater than zero as claimed. 
\end{proof}

\noindent We deduce that the series (\ref{fifteen}) defines an analytic function in some neighbourhood of $z_0=0$. Let us find an explicit (and simple) formula to $f(z)$. For this we first write in five rows the expressions for $f(z),\; c_{0}zf(z),\;  c_{1}z^{2}f(z)\; c_{2}z^{3}f(z)$ and $c_{3}z^{4}f(z)$. We shall write it in a friendly manner to our purpose.

\begin{equation} \label{sixteen}
\begin{matrix}
f(z)= &b_{0}+  &b_{1}z+  &b_{2}z^2+  &b_{3}z^3+  &b_{4}z^4+  &b_{5}z^5+... \\ 
c_{0}zf(z)= &  &c_{0}b_{0}z+  &c_{0}b_{1}z^{2}+  &c_{0}b_{2}z^{3}+  &c_{0}b_{3}z^{4}+  &c_{0}b_{4}z^{5}+... \\ 
c_{1}z^{2}f(z)= &  &  &c_{1}b_{0}z^{2}+   &c_{1}b_{1}z^{3}+  &c_{1}b_{2}z^{4}+  &c_{1}b_{3}z^{5}+... \\ 
c_{2}z^{3}f(z)= &  &  &  &c_{2}b_{0}z^{3}+  &c_{2}b_{1}z^{4}+  &c_{2}b_{2}z^{5}+... \\ 
c_{3}z^{4}f(z)= &  &  &  &  &c_{3}b_{0}z^{4}+  &c_{3}b_{1}z^{5}+... 
\end{matrix}
\end{equation}
Now, by subtracting the last four equations from the first and getting by (\ref{fourteen}) that the 5'th, 6'th, 7'th,... columns of (\ref{sixteen}) are cancelled, we get:
\begin{equation} \label{seventeen}
\begin{aligned}
& f(z)\left ( 1-c_{0}z-c_{1}z^{2}-c_{2}z^{3}-c_{3}z^{4} \right )=b_{0}+\left ( b_{1}-b_{0}c_{0} \right )z+\\
& \left ( b_{2}-b_{1}c_{0}-b_{0}c_{1} \right )z^{2}+\left ( b_{3}-b_{2}c_{0}-b_{1}c_{1}-b_{0}c_{2} \right )z^{3}\;,
\end{aligned}
\end{equation}
that is $f$ is a rational function, $f(z)=\dfrac{p(z)}{q(z)},\; |p|\leq 3,\;|q|= 4;\; q(z)=-c_{3}z^4-c_{2}z^3-c_{1}z^2-c_{0}z+1$ and $p(z)$ is the right-hand side of (\ref{seventeen}). Since $|q|= 4$ we can find all its roots and get
$ q(z)=-c_{3}(z-a_{1})(z-a_{2})(z-a_{3})(z-a_{4})$ for some $a_{1}, a_{2}, a_{3}, a_{4}\in \mathbb{C}$, not necessarily distinct. Then, as in (\ref{thirteen}), we get by the formulas in (\ref{six}), (\ref{seven}) (note that $q(0)\neq 0$) and by Leibniz rule, an explicit and elementary formulas for the derivatives of $f$ at $z=0$. This gives elementary formulas to the elements of the recursive sequence 
$\left \{ b_n \right \}_{n=4}^{\infty }$ $\left (b_{n}=\dfrac{f^{(n)}(0)}{n!}  \right )$ as we wanted.
We shall not write down here these explicit (but complicated formulas).\\
\noindent In a similar way, we can find explicit formulas for a sequence that is determined by the rule
\begin{equation} \label{eighteen}
b_{n+2}=c_{0}b_{n+1}+c_{1}b_{n}\;, n\geq 0
\end{equation} 
$b_{0}, b_{1}$ are given, $c_{1}\neq 0$ or for a sequence given by
\begin{equation} \label{nineteen}
b_{n+3}=c_{0}b_{n+2}+c_{1}b_{n+1}+c_{2}b_{n}\;, n\geq 0
\end{equation} 
$b_{0}, b_{1}, b_{2}$ are given, $c_{2}\neq 0$
which are simpler than (\ref{fourteen}). For finding explicit formulas to the elements of the sequences (\ref{eighteen}) or (\ref{nineteen}) we have to find the roots of polynomials of degrees $2$ or $3$, respectively.\\
\noindent We remark that it is possible to find explicit formulas for the elements of sequences of the types (\ref{fourteen}), (\ref{eighteen}) and (\ref{nineteen}) also without using our formulas (\ref{six}) and (\ref{seven}). Indeed, suppose for example that from (\ref{fourteen}) we get $q(z)=-c_{3}(z-a_{1})^{2}(z-a_{2})(z-a_{3})$, where $a_{1}, a_{2}, a_{3}$ are distinct. Then we have that
$$f(z)=\dfrac{B_{1}^{(1)}}{z-a_{1}}+\dfrac{B_{2}^{(1)}}{(z-a_{1})^{2}}+\dfrac{B_{1}^{(2)}}{z-a_{2}}+\dfrac{B_{1}^{(3)}}{z-a_{3}}\;,$$
and we can find the $B_{j}^{(i)}$'s similarly to (\ref{twelve}). Then we can find easily any derivative of $f$ at $z_{0}=0$ and by this we can derive explicit and elementary formulas for $b_{n},\;, n \geq 4$. These formulas will be indeed explicit, not depending in some recursion rule, because here, in this case the number of the roots is fixed : $L=3$, as their multiplicities are known: $2, 1, 1$. In other sequences of the type (\ref{fourteen}) it can be $L=4$ with multiplicities $1, 1, 1, 1$; $L=2$ with multiplicities: $3, 1$ or $2, 2$ or  $L=1$ with one multiplicity, $4$. The sequences (\ref{eighteen}) and (\ref{nineteen}) can be treated similarly.  

\noindent Let us discuss the general case.

\noindent Let $k \geq 1$ and let $b_{0}, b_{1},...,b_{k}$ be given. Let also $c_{0}, c_{1},...,c_{k}$ be constants with $c_{k} \neq 0$ and let the sequence $\left \{ b_n \right \}_{n=k+1}^{\infty }$ be defined by the recursive rule 
$b_{n+k+1}=c_{0}b_{n+k}+c_{1}b_{n+k-1}+...+c_{k-1}a_{n+1}+c_{k}a_{n},\; n\geq 0$. We showed how to get an explicit formula for $\left \{ b_n \right \}_{n=k+1}^{\infty }$ for the cases $k+1=2, 3, 4$. In view of the ability of using formulas from (\ref{six}) and (\ref{seven}), the \textit{'only'} barrier to get explicit
formulas to the cases $k+1=5, 6, 7, ...$, is the lack of ability of finding the zeros of general polynomial of degree at least $5$. However, we can fix in advance the constants $c_{0}, c_{1},...,c_{k}$ to be the coefficients of some polynomial that we know its zeros. Specifically, if $q(z)=A(z-a_{1})(z-a_{2})\cdots(z-a_{k+1}),\; k+1 \geq 5$ and assume for simplicity that the free coefficient, $A\cdot a_{1}\cdot a_{2}\cdots a_{k+1}(-1)^{k+1}$ is $1$, then we have $q(z)=1-c_{0}z-c_{1}z^{2}-...-c_{k}z^{k+1}$, with constants $c_{j}$'s that we can calculate. Now, let $p(z)$ be the polynomials of degree $k$, which is analogous to the previous $p(z)$ that is the right-hand side of (\ref{seventeen}). Then by the formulas from (\ref{six}) and (\ref{seven}) we get the derivatives of $f(z)$ at $z_{0}=0$ and these will give us the elements of the sequence $\left \{ b_n \right \}_{n=k+1}^{\infty }$.

\subsection{Applications by computers}

\noindent  As we now have explicit and elementary formulas for derivatives of any order of any rational function for which we know the zeros of its denominator (especially in view of remark 3.6, that gives explicit values for $s_{1}, s_{2},...,s_{L},s$, whose we need for calculating the formulas from (\ref{six}) and (\ref{seven})), it is simpler now to write a computer program, that calculates the derivatives of any such function or to write a program that integrates such functions. Moreover, for the case of the general interpolation problem (see subsection 6.1), the situation is even better. We need there only to find derivative of functions that are given in the form (\ref{five}), (the zeros of the denominators there are the interpolation points $a_{1}, a_{2},...,a_{L}$), so for this case, applying Theorem 3.1 by computer is straightforward. 
%%%%%%%%%%%%%%%%%%%%%%%%%%%%%%%bibliography%%%%%%%%%%%%%%%%%%%%%%%%

\end{document}